\documentclass{article}
\usepackage{amsfonts}
\usepackage{graphicx}
\usepackage[utf8]{inputenc}
\usepackage[english]{babel}

\newtheorem{theorem}{Theorem}

\newenvironment{proof}[1][Proof]{\noindent\textbf{#1.} }{\ \rule{0.5em}{0.5em}}
\begin{document}

\title {The Geometry of the Mappings by General Dirichlet Series}

\bigskip

\author {Dorin Ghisa}

{York University, Glendon College, Toronto, Canada}

{dghisa@yorku.ca}

\maketitle

\textit{We dealt in a series of previous publications with some geometric aspects of the mappings by functions obtained as analytic continuations to the whole complex plane of general Dirichlet series. Pictures illustrating those aspects contain a lot of other information which has been waiting for a rigorous proof. Such a task is partially fulfilled in this paper, where we succeeded among other things, to prove a theorem about general Dirichlet series having as corollary the Speiser's theorem. We have also proved that those functions do not possess multiple zeros of order higher than $2$ and the double zeros have very particular locations. Moreover, their derivatives have only simple zeros. With these results at hand, we revisited GRH for a simplified proof.}

\bigskip

\textit{Keywords: general Dirichlet series, $S_{k}$-strips, intertwining curves, fundamental domains, Riemann Hypothesis}

\bigskip

2010 AMS Subject Classification: 30C35, 11M26

\bigskip 

\bigskip

\textbf{\textbf{1. Introduction} } \bigskip

The study of general Dirichlet series has its origins in the works of
(E. Cahen, 1894,1918), (J. Hadamard, 1896, 1908), with important contributions of (E. Landau, 1907,1909,1917,1921), (H. Bohr, 1913, a,b,c,d), (G. H. Hardi, 1911), (G. H. Hardi and M. Riesz, 1915), (T. Kojima, 1914, 1916), (M. Kuniyeda, 1916), (G. Valiron, 1926), etc.
A lot of contemporary mathematicians created a diversified theory of general Dirichlet series, some insisting on the connection with the Laplace-Stieltjes transforms (J. Yu, 1963, Y. Kong and S. Daochun, 2008, X. Luo and Y. Y. Kong, 2012, Y. Gingjing, 2012, Y. Y. Kong and Y. Yan, 2014, C. Singhal and G. S. Srivastava, 2015a and 2015b), others endowing them with some topological structures (P.K. Kamthan and S. K. Shing Gautam, 1975,  L.H. Khoi, 2015), extending to them the Nevanlinna theory (too many to be cited) or dealing with vector valued Dirichlet series (A. Defant, D. Garcia. M. Maestre, D. Perez-Garcia, 2008, J. Bonet, 2015, B.L. Srivastava, 1983, A. Sharma and G. S. Srivastava, 2011 and 2016)  

 We will use normalized series defined as follows.
To any sequence $A=(1=a_{1},a_{2},...)$ of complex numbers and any
increasing sequence $\Lambda =(0=\lambda _{1}<\lambda _{2}<...)$ such that
 $\lim_{n->\infty }\lambda _{n}=\infty $ we associate the series\bigskip

\qquad $(1)\qquad \zeta _{A,\Lambda }(s)=\Sigma _{n=1}^{\infty
}a_{n}e^{-\lambda _{n}s}$\bigskip

\bigskip

$\Lambda $ is called the \textit{type} of the series (1) and the series
defined by the same $\Lambda $ will be called series of the \textit{same type%
}. When $\lambda_{n}=\ln{n}$, we obtain the ordinary (\textit{\textit{proprement dites}}, Valiron, 1926) Dirichlet series.

Suppose that $A$ and $\Lambda $ are such that the \textit{abscissa of
convergence} (see Hardi-Riesz, 1915 and Valiron, 1926) of the series (1)

\bigskip

$\qquad (2)\qquad \sigma _{c}=\lim \sup_{n->\infty }\frac{1}{\lambda _{n}}%
\log| \Sigma _{k=1}^{n}a_{k}|$ \bigskip $\ $\ is finite.

\bigskip

Then (see Cahen, 1894) $z=\zeta_{A,\Lambda }(s)$ is an analytic function in the
half plane ${Re s}>\sigma _{u}$, where $\sigma_{u}$ is the \textit{abscissa of uniform convergence} of (1), and where $\sigma_{u}$ is at most $\sigma_{c}+D$
with

\bigskip

$\qquad (3)\qquad \ D=\lim \sup_{n->\infty }\frac{log {n}}{\lambda _{n}}$ 

\bigskip

For the ordinary Dirichlet series $D=1$ and it is known that in the case of Dirichlet L-series defined by imprimitive Dirichlet characters  $\sigma_{c}=0$ and $\sigma_{u}=1$, while in the case of primitive characters $\sigma_{c}=\sigma_{u}=0$.

In general, when $D=0$ then $\sigma_{u} = \sigma_{c}$,  and the series (1) is an analytic function in the half plane ${Re s}>\sigma _{c}$.

Suppose that this function can be continued analytically to the whole
complex plane, except possible at $s=1$ which is a simple pole. We keep the
notation $\zeta _{A,\Lambda }(s)$ for the function obtained by analytic
continuation.

\bigskip

With the exception of a discrete set of points from the complex plane, this
function is locally injective, i.e. it maps conformally and hence
bijectively small neighbourhoods of every point onto some domains. Enlarging
these neighbourhoods, the image domains get bigger. How big can they get? The
answer is: they can become the whole complex plane with some slits (see
Ghisa, 2013, 2014, 2016a, 2016b). A region with this property is called \textit{%
fundamental domain} of $\zeta _{A,\Lambda }(s)$.

\bigskip

Our aim is to show that the complex plane can be divided into a countable
number of sets whose interiors are fundamental domains of this function. We
have done this previously for the Riemann Zeta function (Andreian-Cazacu, Ghisa, 2009) as well
as for Dirichlet L-functions (Ghisa, 2016a). To do the same thing for functions
obtained by analytic continuations of general Dirichlet series we made (Ghisa, 2016b) the assumption that (1) satisfies a Riemann type of functional equation. We will show next that such a strong assumption is not necessary and we can obtain the same result by using just elementary properties of the conformal mapping.

\bigskip  

It has been proved (Ghisa, 2016b) that for normalized series (1) the limit
 $\lim_{\sigma ->+\infty }\zeta _{A,\Lambda}(\sigma +it)=1$ is uniform with respect to $t.$ This means that for any $\epsilon >0$ there is $\sigma _{\epsilon }\geq \sigma _{c}$ such that $\sigma >\sigma _{\epsilon }$ implies $|\zeta _{A,\Lambda }(s)-1|$ $<\epsilon
,$ therefore the whole half plane ${Re}s>\sigma _{\epsilon }$ is mapped
by $\zeta _{A,\Lambda }(s)$ into a disc centred at $z=1$ of radius $
\epsilon .$ 

\bigskip

An immediate consequence of this fact is that the abscissa of convergence of a normalized series is less than $+\infty$. Moreover, for $\epsilon <1$ there is no zero of $\zeta _{A,\Lambda }(s)$ in the half plane 
${Re}s>\sigma _{\epsilon }$. Also, if $\zeta _{A,\Lambda }(s)$ satisfies a
Riemann type of functional equation, then there cannot be non trivial zeros
of $\zeta _{A,\Lambda }(s)$ in the half plane ${Re} s< -\sigma
_{\epsilon }$ neither.

\bigskip

\bigskip

\textbf{\ \textbf{2. Pre-Images of Lines and Circles}}

\bigskip

Suppose that, for a point $s_0 $ with $Re s_0 >\sigma _{c} $ the function $%
\zeta _{A,\Lambda }(s)$ has a real value $\zeta _{A,\Lambda }(s_0) > 1$. The
continuation from $s_0$ along the interval $[1,+\infty ),$ is a curve $%
\Gamma_{k}^{\prime}$ such that when  $z=\zeta _{A,\Lambda }(s)$ tends to $1$
, we have that $\sigma$ tends to $+\infty $ on $\Gamma _{k}^{\prime} $, or
there is a point $u_{k,j}$ such that $\zeta _{A,\Lambda }(u_{k,j})= 1$ and
the continuation can be carried along the whole real axis giving rise to a 
curve $\Gamma_{k,j}$. We will show later that when $z=\zeta _{A,\Lambda }(s)$ 
tends to $+\infty $ then $\sigma$ tends to $-\infty $ on $\Gamma _{k}^{\prime }$, or
on $\Gamma _{k,j}$, in other words these curves cannot remain in a right
half plane.

\bigskip

Let us notice first that $\zeta _{A,\Lambda} (s)$ are transcendental
functions and $s=\infty$ is an essential singular point for them. The value $%
z=0$ cannot be a lacunary value for $\zeta _{A,\Lambda }(s)$ since then $%
\frac{1}{\zeta _{A,\Lambda }(s)}$ would be a non polynomial integer function
for which $s=\infty$ is an essential singular point and hence $s=0$ would be
also an essential singular point for $\zeta _{A,\Lambda }(s)$, which is not
true. Then, by the Big Picard Theorem, $\zeta _{A,\Lambda} (s)$ has
infinitely many zeros in every neighbourhood of $s=\infty$.

\bigskip

Given a bounded region of the plane, there is $r>0$ such that the pre-image
of the circle $C_{r}$ centred at the origin and of radius $r$ has only
disjoint components, which are closed curves containing each one a unique
zero of $\zeta _{A,\Lambda} (s)$ belonging to that region. When $r$
increases those curves expand and they can touch one another at some points $v_{k}$. 
These are branch points of the function, since in every neighbourhood of $v_{k}$
the function takes at least twice any value on the image circle. Therefore
the derivative of $\zeta _{A,\Lambda} (s)$ cancels at $v_{k}$. It is obvious that
any zero of the derivative, which is not a zero of the function itself, can
be obtained in this way. Indeed, if $v$ is such a zero, we can take 
$r=|\zeta _{A,\Lambda} (v)|$ and necessarily several components of the pre-image 
of $C_{r}$ will pass through $v$.

\bigskip

What happens with those components of the pre-image of a circle $C_{r}$ when $r=1$ ?
We have proved (Ghisa, 2016b, Theorem 1) that there is at least one unbounded component
of the pre-image of the unit circle. 
That proof did not use the assumption of $\zeta_{A,\Lambda} (s)$ satisfying a Riemann type of functional equation and therefore it is true for any function $\zeta _{A,\Lambda} (s)$.

\bigskip

Let us notice that two curves $\Gamma _{k}^{\prime }$ and $\Gamma
_{l}^{\prime }$ cannot intersect each other. Indeed, if $s_{0}$ would be a
common point of these curves, then when $z=\zeta _{A,\Lambda} (s)$ moves on
the interval $I$ between $z_{0}=\zeta _{A,\Lambda} (s_{0})$ and $1$ the
point $s$ describes an unbounded curve which bounds a domain mapped by $%
\zeta _{A,\Lambda} (s)$ onto the complex plane with a slit alongside the
interval $I$. That domain should contain a pole of $\zeta _{A,\Lambda} (s)$
which is not true. Therefore an intersection point $s_{0}$ of the two curves
cannot exist and consecutive curves $\Gamma _{k}^{\prime }$ and $\Gamma
_{k+1}^{\prime }$ bound infinite strips $S_{k}$. We suppose that $S_{0}$ is
the strip containing the point $s=1$ and for every integer $k$, the curve $%
\Gamma _{k+1}^{\prime }$ is situated above $\Gamma _{k}^{\prime }$.

\bigskip

\begin{figure}
  \includegraphics[scale=1]{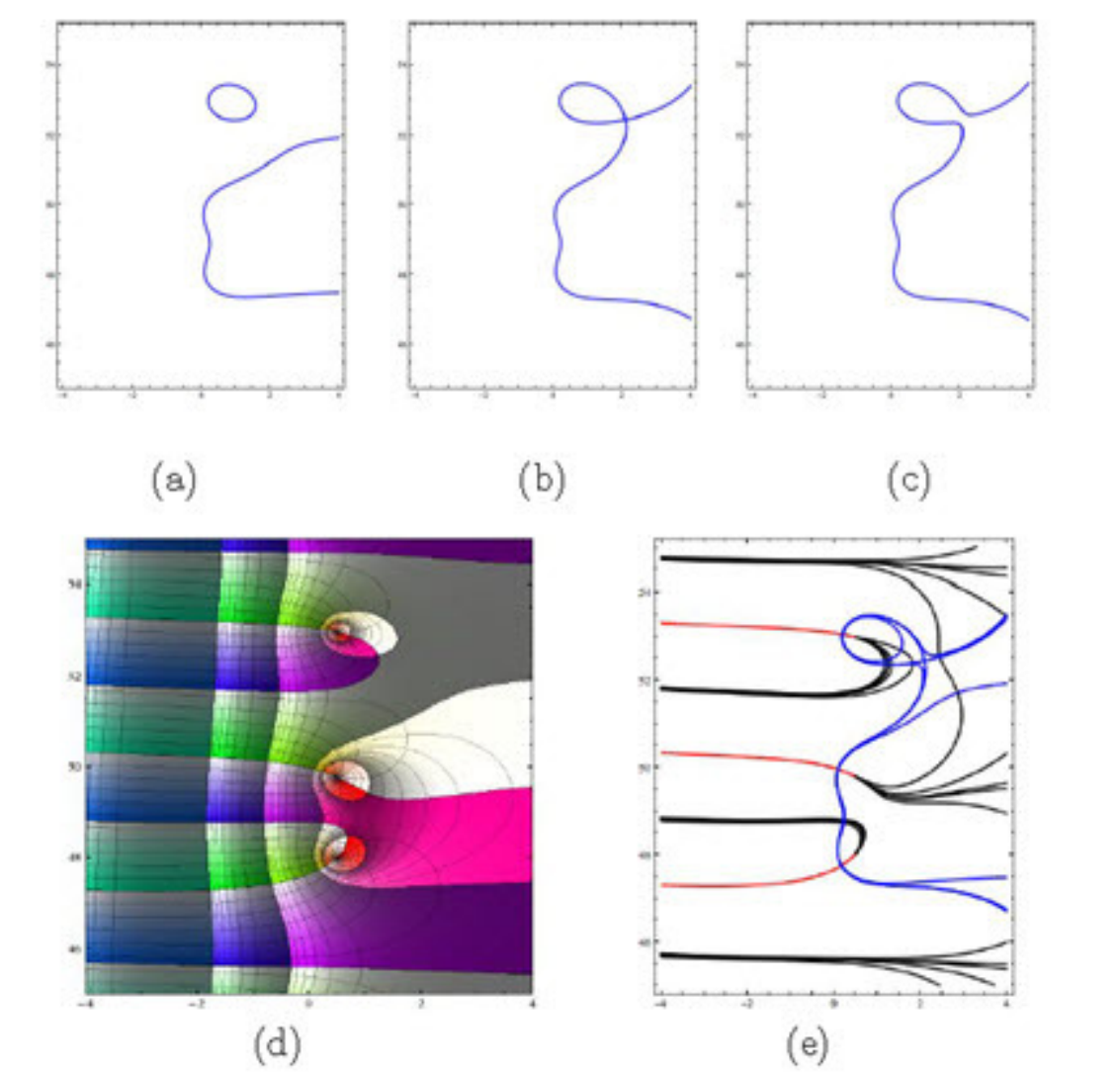}
  \caption{The birth of a strip}
  \label{fig: strips and the unit circle}
\end{figure}

We have also proved (Ghisa, 2016b, Theorem 2) that every unbounded component of the 
pre-image by $\zeta _{A,\Lambda} (s)$ of the unit circle is contained between
two consecutive curves $\Gamma _{k}^{\prime }$ and $\Gamma _{k+1}^{\prime }$
and vice-versa, between two consecutive curves $\Gamma _{k}^{\prime }$ and $%
\Gamma _{k+1}^{\prime }$ there is a unique unbounded component of the
pre-image of the unit circle. It has been shown that the respective component
does not intersect any one of these curves. For $k$ different of $0$ every
strip $S_{k}$ contains also a unique curve $\Gamma_{k,0}$ which is mapped
bijectively by $\zeta _{A,\Lambda} (s)$ onto the interval $(-\infty, 1)$ of
the real axis. There are infinitely many strips $S_{k}$ covering the whole
complex plane (Ghisa, 2016b, Theorem 4). Therefore, there are infinitely many unbounded 
components of the pre-image by $\zeta _{A,\Lambda} (s)$ of the unit circle.
Some strips $S_{k}$ can contain also bounded components of the pre-image of the unit
circle as well as bounded components of the pre-image of $C_{r}$ with $r>1$.

\bigskip

\begin{figure}
  \includegraphics[width=\linewidth]{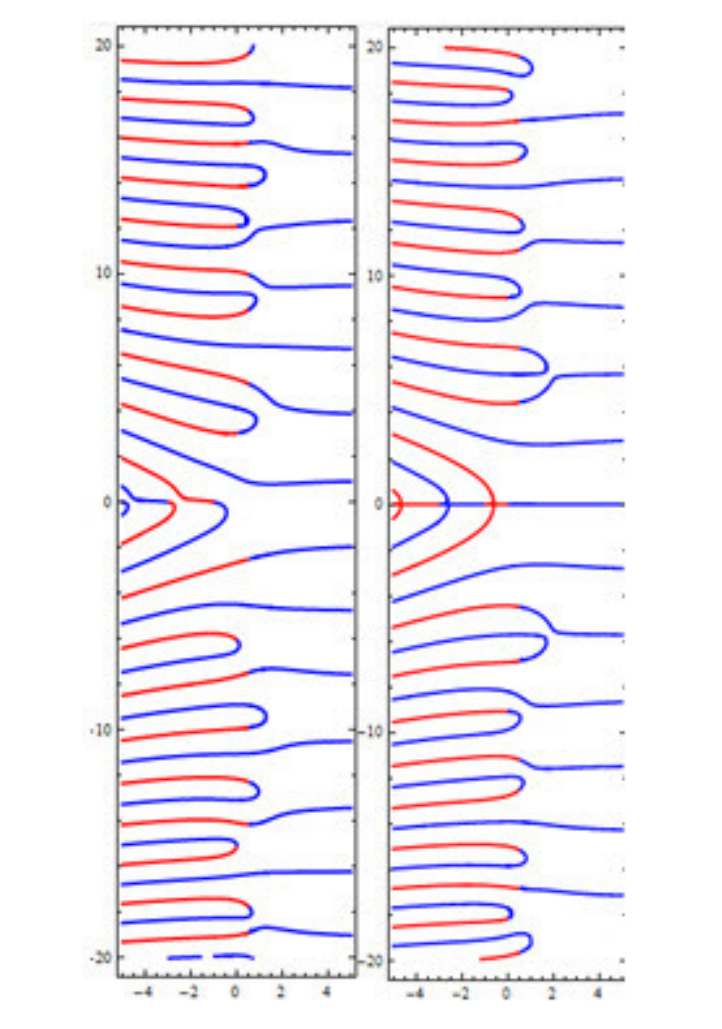}
  \caption{The pre-image of the real axis by  two Ditichlet L- functions}
  \label{fig: imaginary and real Dirichlet characters}
\end{figure}

The use of the pre-image of the real axis can be traced to Speiser's work
(Speiser 1934) on the Riemann Zeta function. After that the pre-image of the
real axis does not appear any more in literature as a tool except for the
paper of Arias-de-Reina (Areas-de-Reina, 2003), who revisits Speiser's theorem. 
In the same year John Derbyshire uses both: the pre-image of the real axis and
that of the imaginary axis in his popular book \textit{Prime Obsession} and
declares that they are \textit{at the heart} of that book. The
classification of the components of the pre-image of the real axis by the Riemann Zeta function appears for the first time in (Andreian-Cazacu, Ghisa, 2009), where the strips $S_k$ are also introduced and a method is devised of partitioning them into fundamental domains. Later, such a classification has been extended to Dirichlet L-functions and finally to functions defined by general Dirichlet series.

\bigskip

A different approach, namely that of phase diagram, has been used by Elias Wegert  for visual exploration of complex functions (Wegert, 2012, 2016). Applied to the Riemann Zeta function, his phase plots revealed interesting patterns pertaining to the universality property of that function. It is known that such a property extends to more general Dirichlet series and probably it can complement our fundamental domains approach.  

\bigskip

\begin{theorem}
No zero of $\zeta _{A,\Lambda} (s)$ or of $\zeta _{A,\Lambda}^{\prime} (s)$
can belong to a curve $\Gamma _{k}^{\prime }$.
\end{theorem}

\bigskip

\begin{proof}
The affirmation of the theorem is obvious for the zeros of $\zeta
_{A,\Lambda} (s)$ since $0$ does not belong to the interval $[1,+\infty)$ of
the real axis. A more intricate argument guarantees that the same is true
for the zeros of the derivative of $\zeta _{A,\Lambda} (s)$. Indeed, even if 
$r>1$ , no bounded component of the pre-image of $C_{r}$ can reach $\Gamma
_{k}^{\prime }$, despite of the fact that $C_{r}$ intersects the interval $%
[1, +\infty)$. Indeed, in the contrary case, the respective component should
intersect $\Gamma _{k}^{\prime }$ at lest twice, or it should be tangent to
it, fact which requires that $C_{r}$ intersects the interval $[1,+\infty)$
the same number of times or to be tangent to it, which is not possible. The
fact that $C_{r}$ intersects the interval $[1,+\infty)$ has as effect the
pre-image of $C_{r}$ intersecting the curves $\Gamma _{k,j}$ (and not $\Gamma
_{k}^{\prime }$ ). It results that no two bounded components of the pre-image
of $C_{r}$ can meet on $\Gamma _{k}^{\prime }$, neither can one of these
components meet on $\Gamma _{k}^{\prime }$ an unbounded component of the
pre-image of $C_{r}$ into a zero of $\zeta _{A,\Lambda}^{\prime} (s)$.

\bigskip

\textbf{Remark}: Theorem 1 does not imply that $\zeta _{A,\Lambda}^{\prime} (s)$ cannot have zeros on some $\Gamma _{k,j}$. Such zeros appeared as possible for the Dirichlet L-function $L(5,2,s)$ as seen in Fig 3 below when $t$ has approximately the values $169.2$ and $179.2$. However, we suspected that this was due to the poor resolution of the picture and indeed, when we zoomed on the respective points, we obtained configurations which show clearly that $\zeta _{A,\Lambda}^{\prime} (s)$ does not cancel there. We will present a rigorous proof of the impossibility of such zeros a little later.

\bigskip
\noindent
 \begin{figure}
   \includegraphics[scale=0.7]{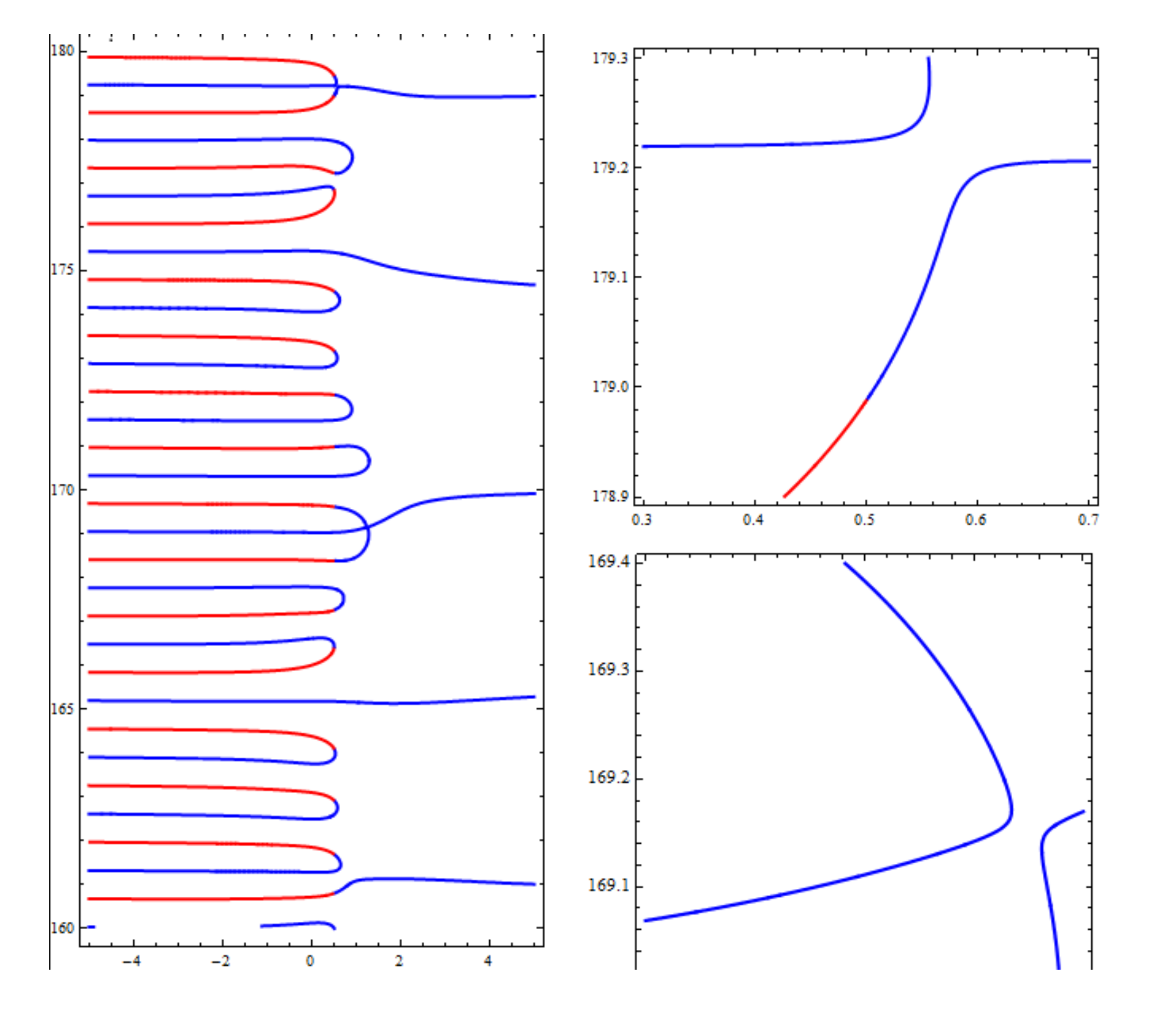}
   \caption{The only zeros of  $\zeta _{A,\Lambda}^{\prime}(s)$ on $\Gamma_{k,j}$      are the double zeros of   $\zeta_{A,\Lambda} (s)$.}
   \label{fig:unexpected intersection}
 \end{figure}

\bigskip

We have seen that the unbounded components of the pre-image of the unit
circle do not intersect any $\Gamma _{k}^{\prime }$. On the other hand the
bounded components of the pre-image of the unit circle intersect curves $%
\Gamma_{k,j}$ at points $u_{k,j}$ where $\zeta _{A,\Lambda} (u_{k,j})$ is $1$
and $-1$. In the same way the bounded components of the pre-image of
$C_{r}=x=_{-}^{+}$1.
with $r>1$ will intersect $\Gamma_{k,j}$ at points $u$ where 
$\zeta _{A,\Lambda} (u)=^+_-r$. However the story of the unbounded components 
of $C_{r}$ is a little more complicated.

\bigskip

When increasing $r$ past $1$ all the unbounded components of $C_{1}$ fuse
together into a unique unbounded curve $\gamma_{r}$ intersecting every curve 
$\Gamma _{k}^{\prime }$, hence they do not generate by this fusion zeros of $%
\zeta _{A,\Lambda}^{\prime}(s)$ . Indeed, since the mapping of $\Gamma
_{k}^{\prime }$ onto the interval $(1,+\infty)$ is bijective, there should
be points $s$ on every $\Gamma _{k}^{\prime }$ such that $\zeta _{A,\Lambda}
(s)=r$. The continuation over $C_{r}$ from each one of these points can be
made clockwise and counter clockwise into $S_{k}$, respectively $S_{k-1}$,
for every $k$, giving rise to that unbounded curve. The final conclusion is
that $\zeta _{A,\Lambda}^{\prime} (s)$ does not cancel on any $\Gamma
_{k}^{\prime }$.\end{proof}

\bigskip

Given any bounded domain in the plane $(s)$, we can take $r>1$ close enough
to $1$ such that $\gamma_{r}$ does not touch any bounded component of the
pre-image of $C_{r}$ included in that domain. However, for bigger values of $r
$ the curve $\gamma_{r}$ comes into contact with every bounded component,
which was turning around one or several zeros $s_{k,j}$ of $\zeta
_{A,\Lambda} (s)$, fusing with it and getting to the left of those zeros,
hence intersecting also the curves $\Gamma_{k,j}$ which contain the
respective zeros. The curve $\gamma_{r}$ is orthogonal to every component of
the pre-image of the real axis if it intersects that component at a point
where $\zeta _{A,\Lambda}^{\prime} (s)$ does not cancel.

\bigskip

We have seen that $\zeta _{A,\Lambda}^{\prime}(s)$ never cancels on 
$ \Gamma_{k}^{\prime }$, yet the question remains if it can cancel on some $\Gamma_{k,j}$. This seemed to be possible as seen by Fig.3. However, as we will see next, bounded components of the pre-image of $C_{r}$ cannot fuse on a curve $\Gamma_{k,j}$ at a zero of $\zeta_{A,\Lambda}^{\prime}(s)$ and the unbounded curve $\gamma_{r}$ cannot fuse on $\Gamma_{k,j}$ with a bounded component of the pre-image of $C_{r}$. 
 
 \bigskip

Since a point turning around the origin in the same direction on an
arbitrary circle $C_{r}$ centred at the origin will meet consecutively the
positive and the negative real half axis, the components of the pre-image of 
$C_{r}$ (including $\gamma_{r}$ when $r>1$) should meet consecutively the pre-image of
the positive and the negative half axis (coloured differently). This is (Ghisa 2014, 2016a) the so-called \textit{color alternating rule}. A zero $s_{0}$ of $\zeta _{A,\Lambda}^{\prime}(s)$ being  on  $\Gamma_{k,j}$ and not being a zero of $\zeta_{A,\Lambda}(s)$ (where the colors are changing) contradicts the color alternating rule since in a neighbourhood of $s_{0}$ the pre-image of $C_{r}$ would meet consecutively the same color. Therefore $\zeta _{A,\Lambda}^{\prime}(s)$ cannot have any zero $s_{0}$ on some $\Gamma_{k,j}$, unless $s_{0}$ is a double zero of  $\zeta_{A,\Lambda}(s)$.
 This remark will help us to state an important generalization of Speiser's theorem (Speiser, 1934) which established an equivalent form of the Riemann Hypothesis.

\bigskip

\begin{figure}
  \includegraphics[width=\linewidth]{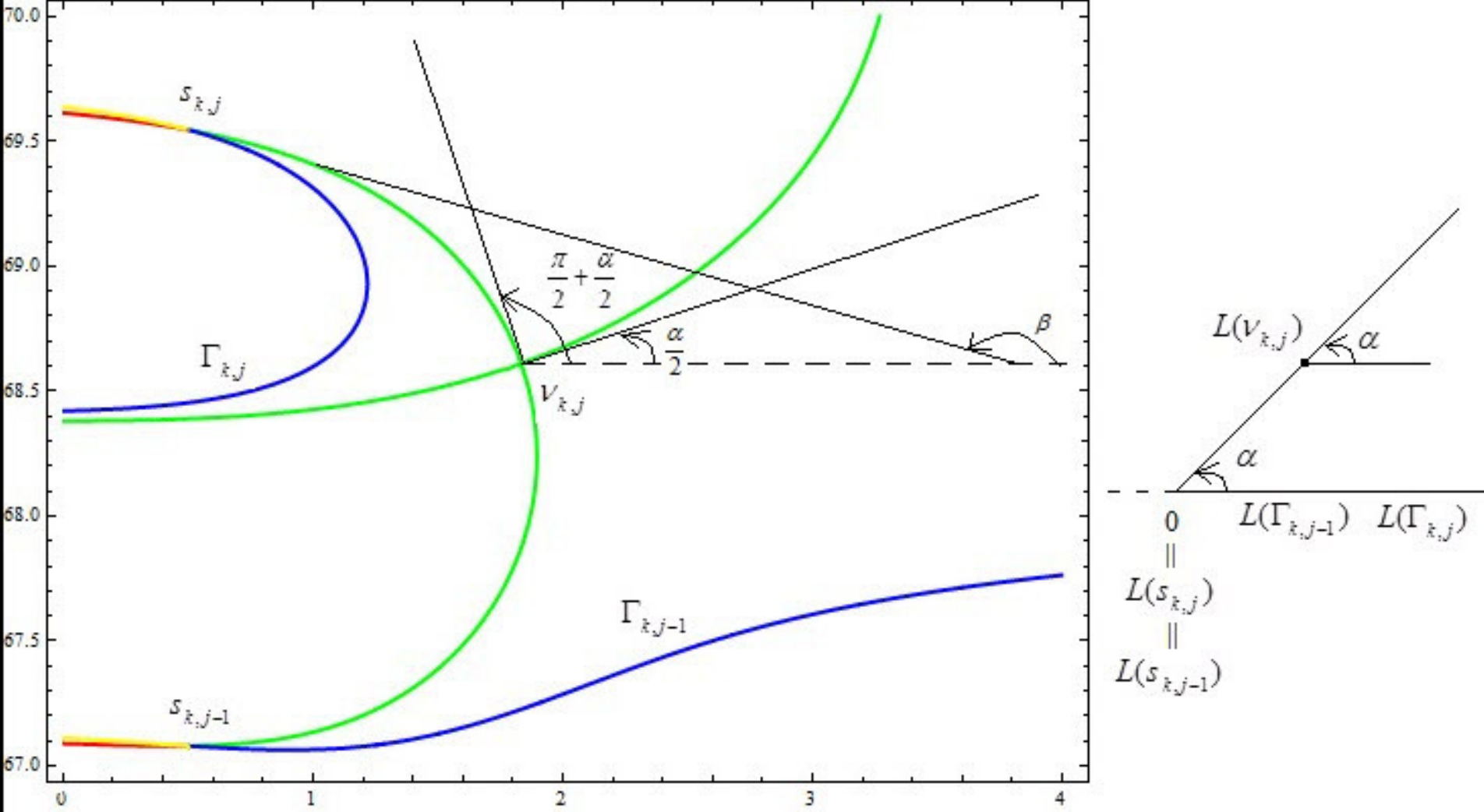}
  \caption{The location of the zeros of the derivative }
  \label{fig: zeros of the derivative}
\end{figure}

\begin{theorem}
For $k$ different of $0$, there is no zero of the derivative of $\zeta
_{A,\Lambda}^{\prime} (s)$ at the left of the leftmost zero of $\zeta
_{A,\Lambda}(s)$ in $S_{k}$
\end{theorem}

\bigskip

\begin{proof}
We can reproduce the proof of Theorem 1 from (Barza, Ghisa, Muscutar, 2014) for arbitrary functions $\zeta _{A,\Lambda}(s)$. Let $s_{k,j}$ be the leftmost zero of $\zeta_{A,\Lambda}(s)$ from $S_{k}$ and suppose that the simple zero $v_{k,j}$ of $\zeta_{A,\Lambda}^{\prime} (s)$ is a progenitor of $s_{k,j}$. It is obviously
enough to deal with the case where $j>0$, i.e. where $\Gamma_{k,j}$ is above 
$\Gamma_{k,0}$. Then $\zeta _{A,\Lambda}(v_{k,j})$ belongs to the upper half plane
(see Fig.4) and $0<\alpha<\pi$, where $\alpha= arg\zeta _{A,\Lambda}(v_{k,j})$. The pre-image of the ray determined by $\zeta _{A,\Lambda}(v_{k,j})$ contains two curves which
are orthogonal at $v_{k,j}$. The angles at $v_{k,j}$ are doubled by $\zeta
_{A,\Lambda}(s)$, hence the four arcs of the pre-image of that ray make
angles of $\alpha/2$, $\pi/2 +\alpha/2$, $\pi +\alpha/2$ and $3\pi/2+\alpha/2
$ with a  horizontal line whose image passes through $\zeta _{A,\Lambda}(v_{k,j})$. The angle $\pi/2+\alpha/2$ made by the second arc
(which ends in $s_{k,j}$) with this line is less than the angle $\beta$ made
by the tangent to the pre-image of the respective ray at any point between $%
s_{k,j}$ and $v_{k,j}$ with the same horizontal line. If $Rev_{k,j}<Res_{k,j}$, then there must be a point on that arc for which $\beta=\pi/2$, therefore $%
\alpha<0$, which is absurd. \end{proof}

\bigskip

\textbf{Corollary}: If $\zeta _{A,\Lambda}(s)$ satisfies RH, then the zeros of $\zeta
_{A,\Lambda}^{\prime} (s)$ from every strip $S_{k}$ are at the right side of
the critical line. When $\zeta _{A,\Lambda}(s)$ is the Riemann Zeta function this corollary represents the Speiser's theorem (Speiser, 1934).

\bigskip

The existence of multiple zeros of functions obtained by analytic
continuations of Dirichlet series has been documented (probably for the
first time) in (Cao-Huu, Ghisa, Muscutar, 2016), where double zeros of a linear combination of Dirichlet L-functions have been found (see Fig. 5 below).

\bigskip

 We have shown that all those double zeros are located on the critical line. In this example the function is $L(7,2,s)+0.34375L(7,4,s)$ and the double zero is obtained for the approximate value of $s$ of $0.5+31.6i$. The double zeros we have found for all the functions of this type were located at the intersection of $\Gamma_{k,0}$ and $\Gamma_{k,1}$ or of $\Gamma_{k,0}$ and $\Gamma_{k,-1}$. We can make now a much more general affirmation about the multiple zeros of functions obtained by analytic continuation of general Dirichlet series.

\bigskip

\begin{figure}
  \includegraphics[width=\linewidth]{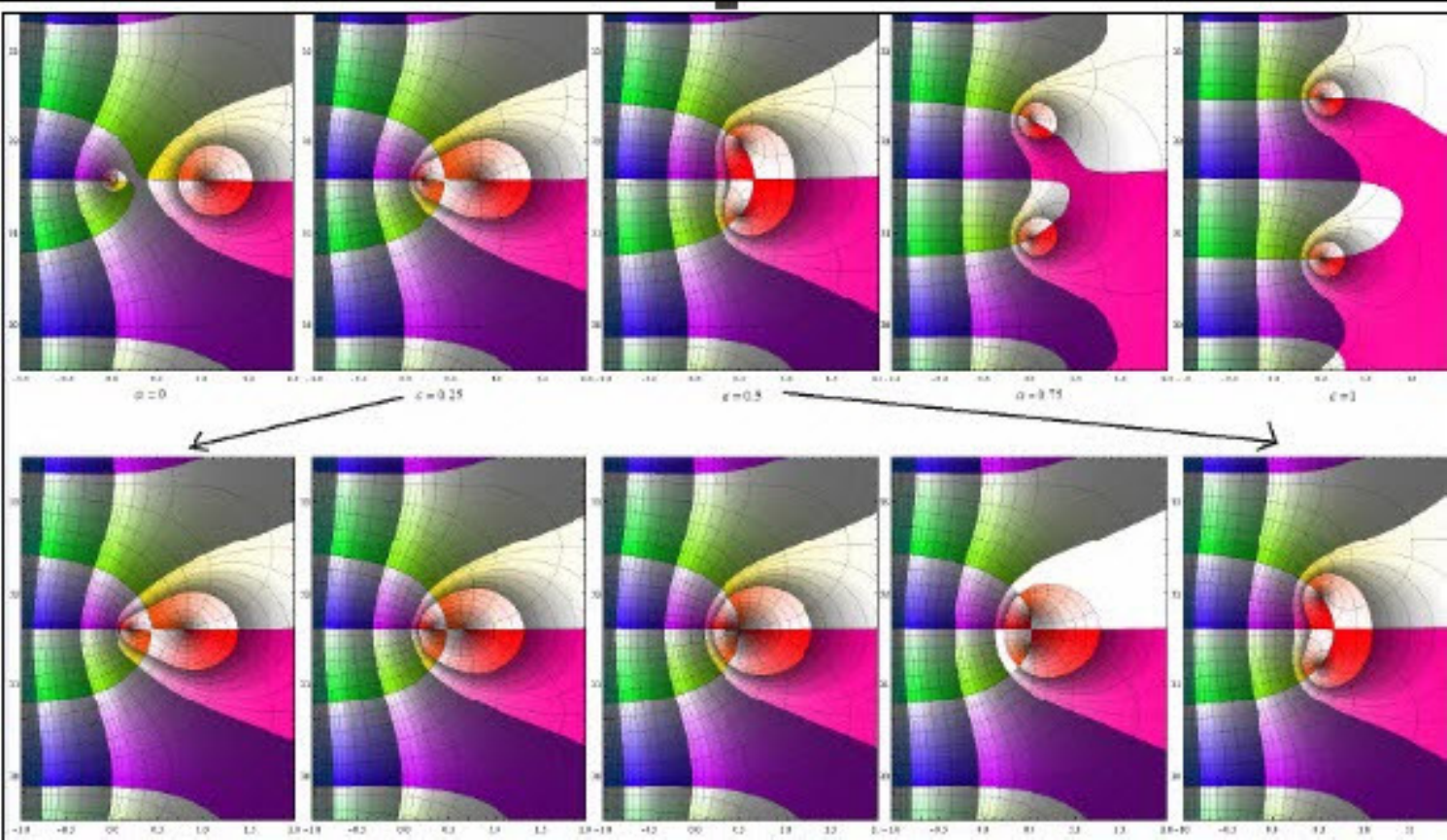}
  \caption{A double zero of a linear combination of Dirichlet L-functions}
  \label{fig:double zero}
\end{figure}

\bigskip

\begin{theorem} In every strip $S_{k}$ of a function $\zeta _{A,\Lambda}(s)$ this function has at most one double zero. Such a zero is found at the intersection of $\Gamma_{k,0}$ and $\Gamma_{k,1}$ or of $\Gamma_{k,0}$ and $\Gamma_{k,-1}$. There is no multiple zeros of $\zeta _{A,\Lambda}^{\prime}(s)$ in $S_{k}$ and hence no zero of a higher order of $\zeta _{A,\Lambda}(s)$. \end{theorem}

\bigskip

We need to postpone the proof of this theorem for a while.

\bigskip

\bigskip

\textbf{\textbf{3. Intertwining Curves}}

\bigskip

When studying functions $z=\zeta _{A,\Lambda }(s)$ it is useful to consider
besides the planes $(s)$ and $(z)$ also a plane $(w),$ where $w=\zeta
_{A,\Lambda }^{\prime }(s).$ Sometimes the planes $(z)$ and $(w)$ will be
identified in order to make more obvious certain relations between the
configurations defined by the two functions in the respective planes. The
configurations we have in view are pre-images by both $\zeta _{A,\Lambda }(s)$
and by $\zeta _{A,\Lambda }^{\prime }(s),$ of some curves or domains.

\bigskip

Regarding the pre-image by $\zeta_{A,\Lambda }(s)$ and by $\zeta _{A,\Lambda
}^{\prime }(s)$ of the real axis it has been found (see Ghisa, 2014, 2016b)
 that the components of these pre-images are paired in such a way that only
the components of the same pair can intersect each other. The respective
pairs form the so-called $\mathit{intertwining}$ $\mathit{curves}$.

\bigskip

Three kinds of intertwining curves have been distinguished (Ghisa, 2013, 2014, 2016a, 2016b), namely:

(a)\qquad\ $\Gamma _{k}^{\prime }$ and $\Upsilon _{k}^{\prime },$ $k\in 
\mathbb{Z}
,$ which are mapped bijectively by $\zeta _{A,\Lambda }(s),$ and by $\zeta
_{A,\Lambda }^{\prime }(s)$ onto the interval $(1,+\infty ),$ respectively $%
(-\infty ,0),$

(b)\qquad $\Gamma _{k,0}$ and $\Upsilon _{k,0},$ $k\in 
\mathbb{Z}
,$ which are mapped bijectively by $\zeta _{A,\Lambda }(s),$ and by $\zeta
_{A,\Lambda }^{\prime }(s)$ onto the interval $(-\infty ,1),$ respectively $%
(0,+\infty )$

(c)\qquad $\Gamma _{k,j}$ and $\Upsilon _{k,j},$ $j\neq 0,$ $k\in 
\mathbb{Z}
,$ $j\in J_{k}\backslash \{0\}=$ (a finite set of integers), which are
mapped bijectively by $\zeta _{A,\Lambda }(s),$ respectively by $\zeta
_{A,\Lambda }^{\prime }(s)$ onto the whole real axis.\bigskip

\bigskip

\begin{theorem}
The intertwining curves touch each other at the points where the tangent to $%
\Gamma _{k,j}$, respectively $\Gamma _{k}^{\prime }$ is horizontal.
Vice-versa, if at a point of such a curve the tangent is horizontal, then a
component of the pre-image of the real axis by $\zeta _{A,\Lambda
}^{\prime}(s)$ passes also through that point.
\end{theorem}

\bigskip

\begin{proof}
Indeed, suppose that $s=s(x)$ is the equation of a curve $\Gamma
_{k}^{\prime }$ or of a curve $\Gamma _{k,j} $ such that $\zeta _{A,\Lambda
}(s(x))=x$. Then

\bigskip

$\qquad (4) \qquad\ \zeta_{A,\Lambda}^{\prime}(s(x) )s^{\prime}(x)=1$.

\bigskip

The equation (4) shows that the argument of $\zeta _{A,\Lambda
}^{\prime}(s(x))$ is opposite to that of $s^{\prime}(x)$, therefore they
cancel simultaneously, and when one is ${\pi}$, the other should be $-{\pi}$%
. Yet these values of the argument of a point means that the respective
point is on the real axis, therefore $s(x)$ belongs to both pre-images of the
real axis, by $\zeta _{A,\Lambda }(s)$ and $\zeta _{A,\Lambda }^{\prime}(s)$ which
completely proves the theorem. \end{proof}

\bigskip

\textbf{Remark}: The Theorem 4 is a corollary of a much more general property which says that if $z=f(s)$ is an analytic function in a domain $D$ of the complex plane and $\gamma$ is the image by $f(s)$ of a smooth curve $\Gamma: s=s(z)$, then denoting by $\Upsilon$ the pre-image by $f^{\prime}(s)$ of $\gamma$, at every point $s_{0}$ where $\Gamma$ and $\Upsilon$ intersect each other we have $argf^{\prime}(s_{0}) + args^{\prime}(z_{0})=0(mod 2\pi)$.

\bigskip

Indeed, suppose that the planes $(z)$ and $(w)$ are identified, where $w=f^{\prime}(s)$ and write $S=S(w)$ for the curve $\Upsilon$. Then at an intersection point $s_{0}$ of $\Gamma$ and $\Upsilon$ we have $s_{0}=s(z_{0})=S(f^{\prime}(s_{0}))$ and $f(s(z))=z$ implies $f^{\prime}(s(z))s^{\prime}(z)=1$, hence $f^{\prime}(s_{0}))s^{\prime}(z_{0})=1$, etc.

\begin{figure}
  \includegraphics[width=\linewidth]{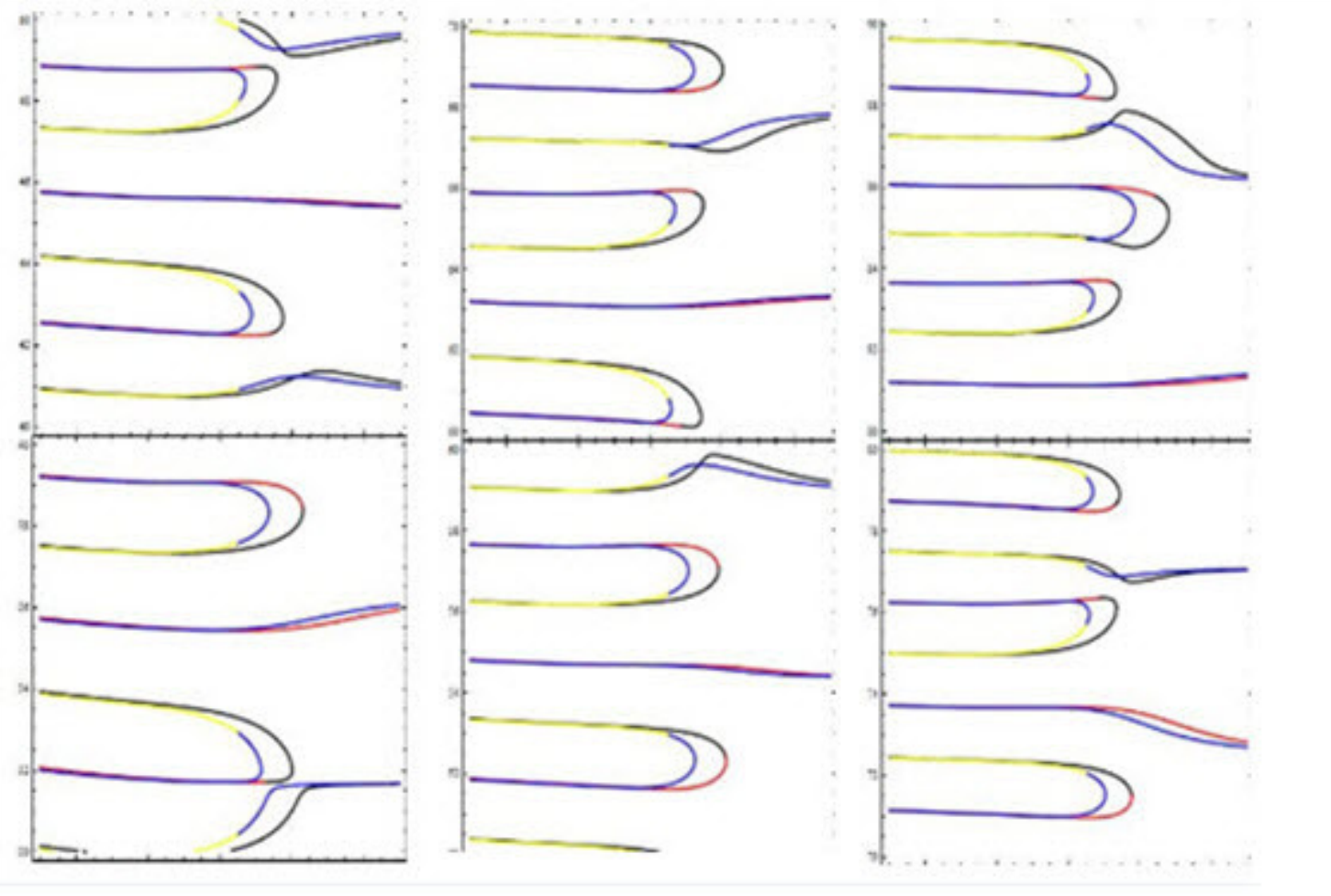}
  \caption{Intertwining curves}
  \label{fig:i.c.}
\end{figure}

\bigskip

We can show now that:

\bigskip

\begin{theorem}
No strip $S_{k}$ can be included in a right half plane.
\end{theorem}

\bigskip

\begin{proof}
The formula (4) written for $\Gamma _{k}^{\prime }$ tells us that when $%
s^{\prime}(x)$ tends to $\infty$ then $\zeta _{A,\Lambda }^{\prime}(s(x))$
must tend to $0$. Yet, there is no zero of $\zeta _{A,\Lambda
}^{\prime}(s(x))$ on any $\Gamma _{k}^{\prime }$. On the other hand,  $\zeta_{A,\Lambda }^{\prime}(s(x))$  is an unlimited continuation of $\zeta _{A,\Lambda}^{\prime}(s)$ alongside  $\Gamma _{k}^{\prime }$ , hence when $s(x)$  tends to $\infty$ we have that  $\zeta_{A,\Lambda }^{\prime}(s(x))$  tends to $\infty$  not to $0$ and this is a contradiction. The conclusion is that the geometry of the pre-image of the real axis is in the whole complex plane similar to that we can see in a bounded region of the plane in the figures 2,3,6,7.\end{proof}

\bigskip

Summarizing these facts and having in view Ghisa 2016b we can say:

\bigskip

\begin{theorem}
The variable $\sigma$ takes any real value on every curve $\Gamma
_{k}^{\prime }$ . Consecutive curves $\Gamma _{k}^{\prime }$ form infinite
strips $S_{k}$ which are mapped (not necessarily bijectively) onto the whole
complex plane with a slit alongside the interval $[1,+\infty)$ of the real
axis. If $S_{k}$ contains $m_{k}$ zeros of the function $\zeta _{A,\Lambda
}(s)$, then it will contain $m_{k}-1$ zeros of $\zeta _{A,\Lambda }^{\prime
}(s)$. It also contains a unique unbounded component of the pre-image of the
unit circle and a unique component of the pre-image of the real axis which is
mapped bijectively by $\zeta _{A,\Lambda }(s)$ onto the interval $(-\infty,1)
$ as well as $m_{k}-1$ components  $\Gamma _{k,j}$ of the pre-image of the
real axis which are mapped bijectively onto the whole real axis. If $\zeta
_{A,\Lambda }(s)$ satisfies a Riemann type of functional equation, every
strip $S_{k}$, $k\neq {0}$ contains a finite number of zeros of this
function. The strip $S_{0}$ may contain infinitely many zeros of $\zeta
_{A,\Lambda }(s)$. \end{theorem}

\bigskip

\begin{figure}
  \includegraphics[width=\linewidth]{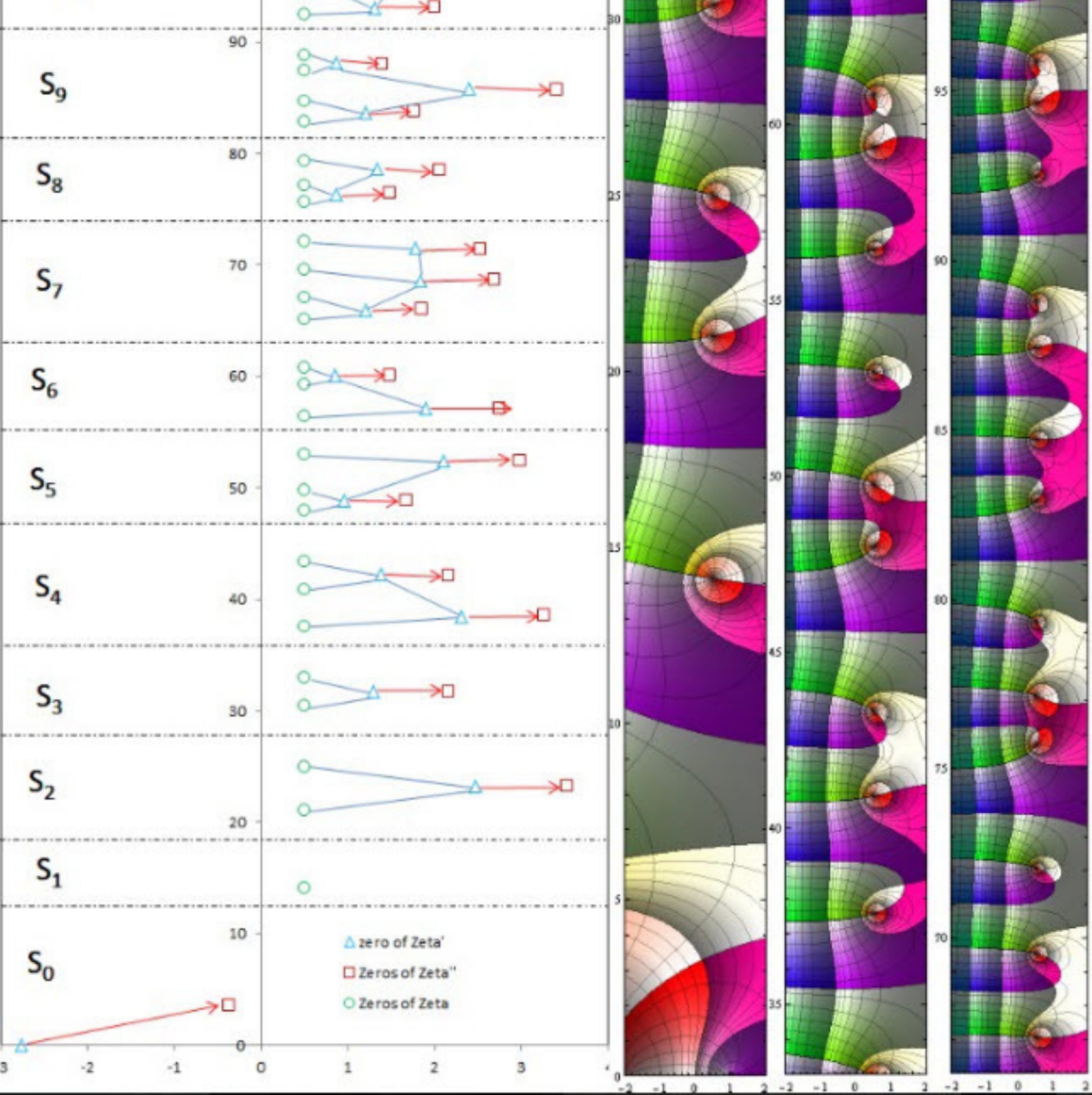}
  \caption{The zeros of the function and of its first two derivatives}
  \label{fig: trees}
\end{figure}

\begin{proof} We only need to justify the numbers $m_{k}$ and $m_{k}-1$ to which the
theorem makes reference. We have seen that every zero of $\zeta _{A,\Lambda
}^{\prime}(s)$ is obtained when two components of the pre-image of some
circle $C_{r}$ centred at the origin and of radius $r$ come into contact.
If we consider the zeros of $\zeta _{A,\Lambda }(s)$ as the liefs of a binary
tree whose internal nodes are obtained in this way, that tree is a complete
binary tree and it is known that it must have exactly $m_{k}-1$ internal
nodes.
\end{proof}

\bigskip

We can prove now also the Theorem 3.

\bigskip

\begin{proof}
Exactly $2$ curves $\Gamma _{k,j}$ and $\Gamma _{k,j\prime}$ should meet at a double zero $s_{0}$ of $\zeta _{A,\Lambda }(s)$ belonging to  $S_{k}$. But $s_{0}$ is also a simple zero of $\zeta _{A,\Lambda }^{\prime}(s)$ and therefore a curve, say $\Upsilon _{k,j}$ should pass through $s_{0}$. Yet, $\Gamma _{k,j\prime}$ intertwines with a curve $\Upsilon _{k,j\prime}$
and unless this curve is $\Upsilon _{k,0}$ (which does not contain any zero of
 $\zeta _{A,\Lambda }^{\prime}(s)$ ) the point $s_{0}$ would be a double zero of 
for $\zeta _{A,\Lambda }^{\prime}(s)$, which is absurd.

\bigskip

 Hence , necessarily one of the curves $\Gamma _{k,j}$ passing through $s_{0}$ is $\Gamma _{k,0}$.  This shows that there can be only one double zero in $S_{k}$ and the other curve passing through $s_{0}$ in that case is either $\Gamma _{k,1}$, or $\Gamma _{k,-1}$.
 
\bigskip

We notice that the curve $\Upsilon _{k,0}$ does not contain any zero of  
$\zeta _{A,\Lambda}^{\prime}(s)$, hence it cannot pass through a multiple zero of order $m$ of  $\zeta _{A,\Lambda }^{\prime}(s)$.
 Since every curve passing through that zero has an intertwining curve defined by the second derivative of  $\zeta _{A,\Lambda }(s)$, the respective point should be a multiple zero of order $m$ of this second derivative, which is absurd. Therefore  $\zeta _{A,\Lambda }^{\prime}(s)$ has no multiple zero and then  $\zeta _{A,\Lambda }(s)$ cannot have any zero of the higher order than $2$ .
\end{proof}

\textbf{\textbf{4. Fundamental Domains}}

For $j\neq 0,$ the curves $\Gamma _{k,j}$ as well as $\Upsilon _{k,j}$ are
parabola like curves with branches extending to infinity as $\sigma
->-\infty $. Therefore we can distinguish between the interior and the
exterior of such curves. They can be viewed as oriented curves, with the
same orientation as the real axis whose components of the pre-image they are.
By the same rule $\Gamma _{k}^{\prime }$ and $\Upsilon _{k}^{\prime },$ as
well as $\Upsilon _{k,0}$ are also oriented, the positive orientation of $%
\Gamma _{k}^{\prime }$ and of $\Upsilon _{k,0}$ being from the right to the
left, while that of $\Upsilon _{k}^{\prime },$ is from the left to the
right.  Some curves $\Gamma _{k,j}$, $j\neq {0}$ can contain in interior
some other curves of the same type (\textit{embraced curves}) and by the
color alternating rule the orientation of the \textit{embracing curve} and
that of the embraced curves must be different. We did not find any instance
where an embraced curve is in turn embracing, yet there is no reason to
believe that such a situation is impossible.


\begin{figure}
  \includegraphics[width=\linewidth]{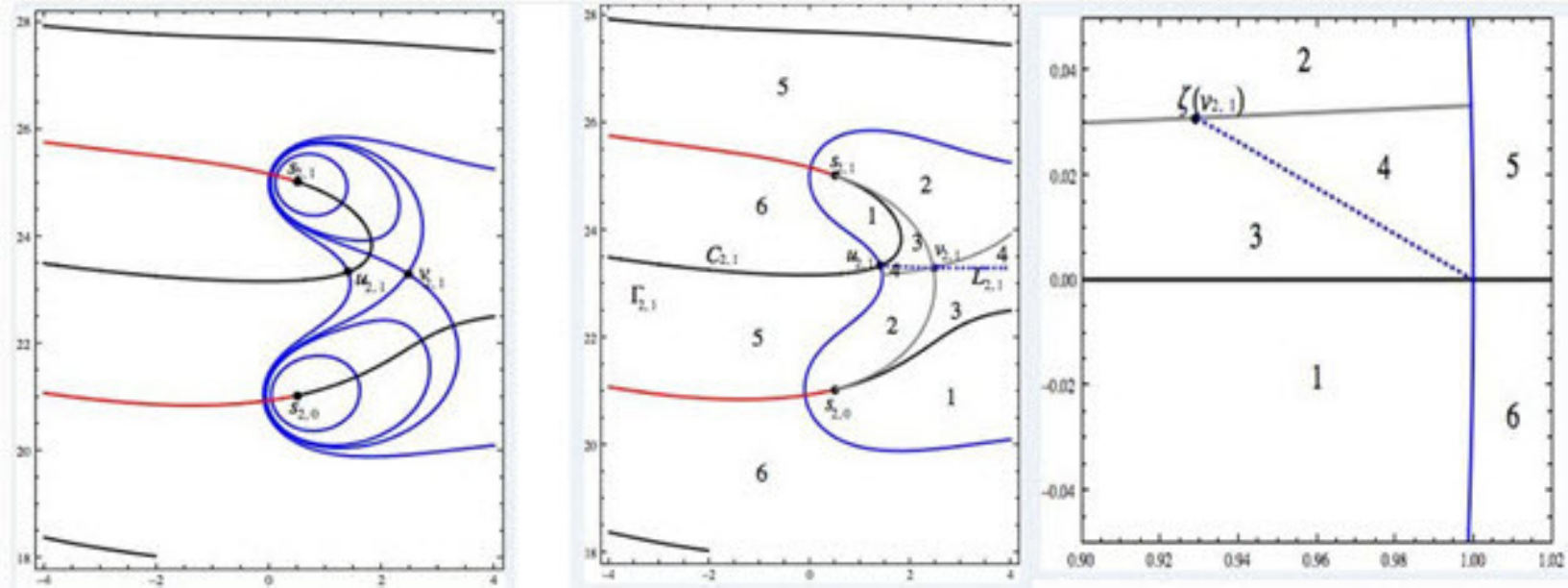}
  \caption{Two fundamental domains and their images}
  \label{fig: f. d.}
\end{figure}

\begin{theorem} If the strip $S_{k}$ contains $m_{k}$ zeros counted with multiplicities, then $S_{k}$ can be partitioned into $m_{k}$ sub-strips which are fundamental domains for  $\zeta _{A,\Lambda }(s)$

\end{theorem}

\begin{proof} Suppose that curves $\Gamma _{k,j}$ and  $\Gamma _{k,j^{\prime}}$ are containing the simple zeros $s_{k,j}$ and $s_{k,j^{\prime}}$ and two components of the pre-image of a circle $C_{r}$ which are going around each one of these zeros touch at a point $v_{k,j}$. This is a zero of $\zeta _{A,\Lambda }^{\prime}(s)$. The pre-image of the segment of line connecting $z=1$ and  $%
\zeta _{A,\Lambda }(v_{k,j}) $ has as component an arc $\eta _{k,j}$
connecting the points on the two curves where $\zeta _{A,\Lambda }(s)=1$
and passing through $v_{k,j}$. If one of these curves is $\Gamma _{k,0}$ then the respective point is $\infty$ and $\eta _{k,j}$ is an unbounded curve. The strip $\Omega _{k, j}$ bounded by this curve and the branches of  $\Gamma _{k,j}$ and $\Gamma _{k,j^{\prime}}$ corresponding to the interval $[1,+{\infty})$,  turns out to be a fundamental domain of $\zeta _{A,\Lambda }(s)$.
Indeed, $\Omega _{k, j}$ is mapped conformally by $\zeta _{A,\Lambda }(s)$
onto the whole complex plane with a slit alongside the interval $[1,+{\infty}%
)$ followed by a slit alongside the segment from $z=1$ to $\zeta _{A,\Lambda
}(v_{k,j}) $. If $\Gamma _{k,j}$ embraces $\Gamma _{k,j^{\prime}}$ then $\Omega _{k, j}$ is bounded to the right. 

\bigskip

Suppose now that one of the two zeros is a double zero of $\zeta _{A,\Lambda }(s)$. We know that $\Gamma _{k,0}$ must pass through that zero and the part of $\Gamma _{k,0}$ corresponding to the interval $[0,1]$ becomes part of the boundary of a fundamental domain obtained in the same way. This time the image domain will be the complex plane with a slit alongside the positive real half axis and the interval from $z=1$ to $z= \zeta _{A,\Lambda }(v_{k,j})$.

\bigskip

We know that  $s_{k,j}$ cannot have a higher order of multiplicity and therefore the cases analysed exhaust all the possibilities. Having in view Theorem 6, we conclude that the strip $S_{k}$ can be always divided into $m_{k}$ fundamental domains.\end{proof}

\bigskip

Fig. 8 above illustrates this theorem for the case of the Riemann Zeta function and the strip $S_{2}$.

As seen in the case of Dirichlet L-functions (including the Riemann Zeta
function) the strip $S_{0}$ has infinitely many zeros, yet following the
same technique we can divide it into infinitely many fundamental domains.

If four different colors are used, say color 1 and 2 for the pre-images by $%
\zeta _{A,\Lambda }(s)$ of the positive and of the negative real half axis
and 3 and 4 for the pre-images by $\zeta _{A,\Lambda }^{\prime }(s)$ of the
same half axes, then two simple topological facts can be established:

a) The\textit{\ alternating color rule}, which states that as a point turns
indefinitely in the same direction on a circle centred at the origin, the
pre-images of this point by each one of the functions $\zeta _{A,\Lambda
}(s) $ and $\zeta _{A,\Lambda }^{\prime }(s)$ will meet alternatively the
colors 1 and 2, respectively 3 and 4. 

b) The \textit{matching color rule}, which states that when intertwining
curves meet each other, then if these are not $\Gamma _{k,0}$ and $\Upsilon_{k,0}$  color 1 will always meet color 4 and color 2 will always meet color 3. Only the curves $\Gamma_{k,0}$ and $\Upsilon_{k,0}$ can intersect each other at points where color 2 can meet color 4.

The series (1) which are Euler products display special important properties.

\bigskip

\textbf{5. Euler Products}

\bigskip

It is known that the Dirichlet L-functions are meromorphic continuations of ordinary Dirichlet series defined by  Dirichlet characters and these series can be expressed as Euler products. This property is a corollary of the fact that the Dirichlet characters are totally multiplicative functions (see for example Ghisa, 2014). Yet the property of being total multiplicative can be extended to general Dirichlet series, as done in Ghisa, 2016a, and therefore some of the general Dirichlet series $\zeta _{A,\Lambda}(s) $ (see the details in Ghisa, 2016a) can also be written as Euler products:

\bigskip

\qquad $(5)\qquad \zeta _{A,\Lambda }(s)=\Sigma _{n=1}^{\infty
}a_{n}e^{-\lambda _{n}s}=\Pi_{p\in \mathbb{P}}  (1-a_{p}e^{-\lambda _{p}s})^{-1}$,

\bigskip

where $\mathbb{P}$ is the set of prime numbers. This convention will be kept in the following for all the products and sums involving the subscript $p$. The product has the same abscissa of  convergence as the series itself. 

\bigskip

Looking for counterexamples to the Grand Riemann Hypothesis (GRH), some Dirichlet series satisfying a Riemann type of functional equation have been found, whose analytic continuation exhibit off critical line non trivial zeros, namely the Davenport and Heilbronn type of functions and linear combinations of L-functions satisfying the same functional equation. Although these are not counterexamples to GRH, their study allowed us to draw interesting conclusions. We have seen in Ghisa, 2016b that if 
$\zeta _{A,\Lambda}(s) $ does not satisfy the GRH, then for every two distinct non trivial zeros
$s_{1}=\sigma+it$ and $s_{2}=1-\sigma+it$ there is $\tau_{0}$, $0<\tau_{0}<1$ such that 
$\zeta _{A,\Lambda}^{\prime}(s(\tau_{0}))=0$, where $s(\tau)=(1-\tau)s_{1}+\tau s_{2}$, i.e. the derivative of $\zeta_{A,\Lambda}(s)$ cancels at a point $s_{0}=s(\tau_{0})$ of the interval $I$ determined by $s_{1}$ and $ s_{2}$. Moreover, $Re\zeta _{A,\Lambda}(s_{0})<1/2$.

\bigskip

Let us rephrase and give a simplified proof to Theorem 3, Ghisa, 2016b.

\bigskip

\begin{theorem} \textit{Suppose that the function (5) satisfies a Riemann type of functional equation and the respective series has the abscissa of convergence $\sigma_{c}<1/2$. 
 Then for every non trivial zero $\sigma +it$ of $\zeta_{A,\Lambda }(s)$  we have $\sigma=1/2$.} \end{theorem}   

\bigskip

\begin{proof} Suppose that there is a zero $\sigma +it$ of  $\zeta_{A,\Lambda }(s)$ for which 
$\sigma>1/2$. Then, due to the functional equation, $1-\sigma+it$ is also a zero of $\zeta_{A,\Lambda }(s)$. There is $r>0$ such that in the components $D_{1}$ and $D_{2}$ of the pre-image of the disc centred at the origin and of radius $r$ containing the respective zeros the function $\zeta_{A,\Lambda }(s)$ is injective. Then we can define the function $\phi:D_{1}->D_{2}$ as follows:
\bigskip\noindent

 \qquad $(6) \qquad \phi(s)= \zeta_{A,\Lambda }(s)^{-1}_{ |D_{2}} \circ \zeta_{A,\Lambda }(s)$

\begin{figure}
  \includegraphics[scale=0.7]{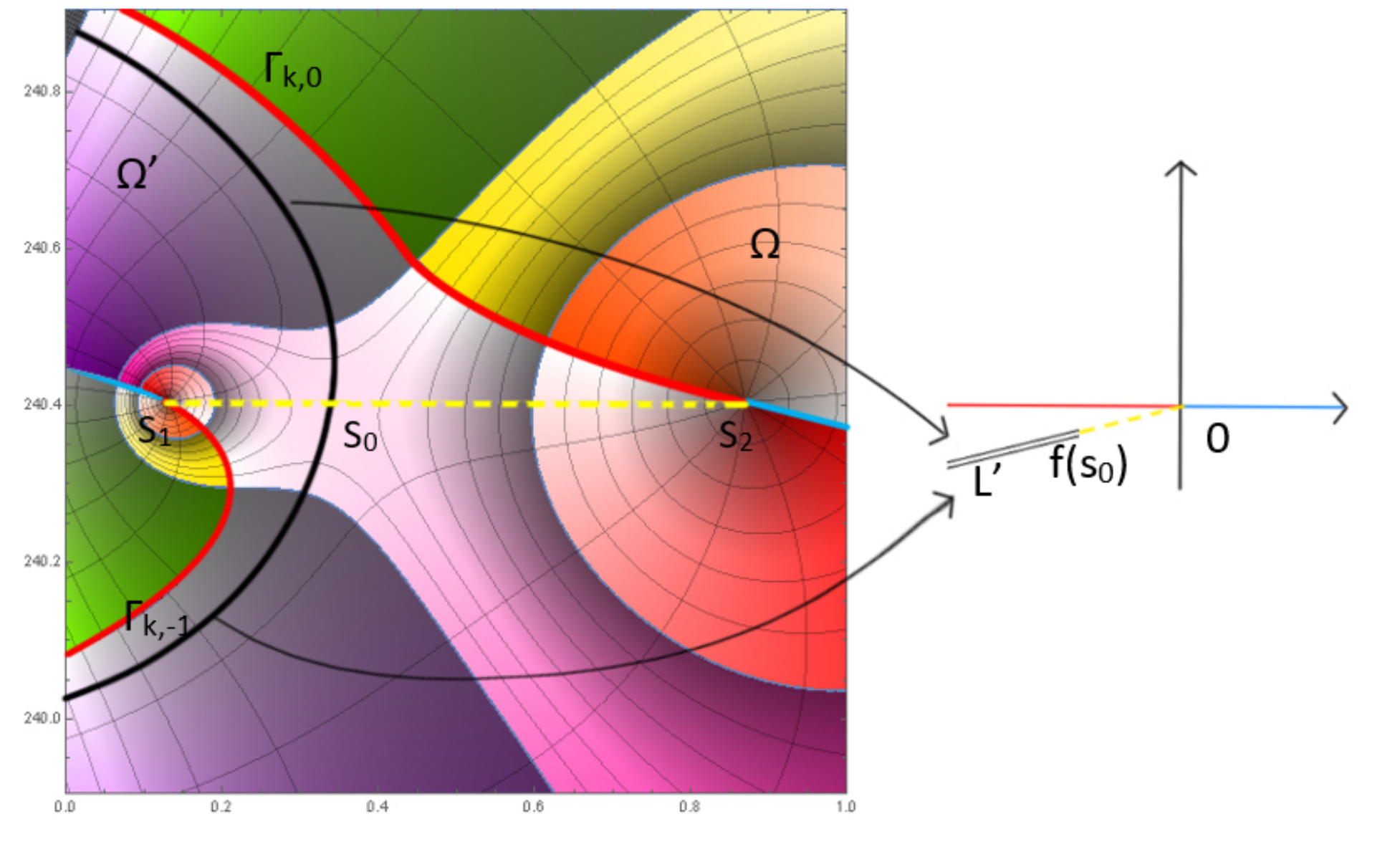}
  \caption{Two fundamental domains and their images}
  \label{fig: f.d2.}
\end{figure}

\medskip\noindent

The function $\phi$ can be continued as an analytic involution of the union
$\Omega \cup \Omega^{\prime}$
the fundamental domains $\Omega$ and $\Omega^{\prime}$ of  $\zeta_{A,\Lambda }(s)$ containing the respective zeros. The boundaries of the domains $\Omega$ and $\Omega^{\prime}$ have a common component $L$ and
the union  $\Omega \cup \Omega^{\prime}\cup L$ is a  simply connected domain $H$. The function $\phi$ can be continued to an analytic involution of $H$ having $s_{0}$ as a fixed point.
We have $\phi(\phi(s))=s$ in $H$, in particular $\phi(\sigma+it)=1-\sigma+it$ and $\phi(1-\sigma+it)=\sigma+it$.
 Moreover, $\zeta_{A,\Lambda }(\phi(s))=\zeta_{A,\Lambda }(s)$ in $H$.
 
\bigskip 

Let us define the function $\Phi$ by
\bigskip

\qquad $(7) \qquad \Phi(s)= \zeta_{A,\Lambda }(s)/\zeta_{A,\Lambda }(\phi(s))$

\bigskip

Since the numerator and the denominator of $\Phi$ are analytic functions in $H$ and the denominator cancels only at  $\sigma+it$ and $1-\sigma+it$, the function $\Phi$ is analytic in $H$ except at these two points.
Since $\zeta_{A,\Lambda }(s)=\zeta_{A,\Lambda }(\phi(s))$, we have that $\Phi(s)=1$ for $s$ not equal to one of these points. Yet they are  removable singularities,
and we can set $\Phi(s)=1$ in $H$. By the formula (5) we have 

\bigskip

\qquad $(8) \qquad \Phi(s)=\Pi_{p=1}^{\infty}(1-a_{p}e^{-\lambda _{p}\phi(s)})/(1-a_{p}e^{-\lambda _{p}s})$

\bigskip

In particular, $\Phi(s)=1$ in $H$ implies $\Phi(1/2 +it)=1$, hence $arg\Phi(1/2+it)=\Sigma_{p=1}^{\infty} arg(e^{i\lambda_{p}t}-a_{p}e^{-\lambda_{p}\delta})/(e^{i\lambda_{p}t}-a_{p}e^{-\lambda_{p}/2})=2k\pi$ for some integer $k$, where $\delta=\phi(1/2)$.

\bigskip

 The terms of this series represent the angles under which the segment between $a_{p}e^{-\lambda_{p}/2}$ and $a_{p}e^{-\lambda_{p}\delta}$ is seen from the point $e^{\lambda_{p}ti}$ on the unit circle.
 If a Ramanujan type condition is fulfilled, namely $A$ and $\Lambda$ are such that for every 
 $\delta>0$ we have $\lim_{n->\infty}|a_{n}|e^{-\lambda_{n}\delta}=0$, then the respective angles tend to zero as $p -> \infty$. This appears to be a necessary condition for the convergence of the series. However, the condition is implicitly
 satisfied  since we know that $\Phi(s)$ is well defined in
 the union of  $\Omega \cup \Omega^{\prime}$.   
  The series is a continuous function of $\delta$, yet it can only take integer multiple values of $2\pi$, which is absurd. Therefore  $\zeta_{A,\Lambda}(s)$ cannot have any non trivial zero $\sigma+it$ with $\sigma$ strictly greater than  $1/2$, which completely proves the theorem.\end{proof}

\bigskip

\textbf{Remark:} In all of our publications we understood by trivial zeros of an L-function those zeros which can be trivially computed. In this respect, the non trivial zeros of the alternating Zeta function $\zeta_{a}(s)=(1-2^{1-s})\zeta(s)$ are the same as those of $\zeta(s)$ and $\zeta_{a}(s)$ satisfies the conditions of Theorem 8, therefore its non trivial zeros are located on the critical line. Thus, RH is also true for $\zeta(s)$. Similarly, the non trivial zeros of a Dirichlet L-function $L(q,1,s)=(1-q^{1-s})\zeta(s)$  defined by the principal character modulo $q$ are the non trivial zeros of the Riemann Zeta function. Also, the  non trivial zeros of a Dirichlet L-function induced by an imprimitive character are the non trivial zeros of the function defined by the associated primitive character. Consequently, the RH for any Dirichlet L-function is fulfilled. With this understanding of the concept of non trivial zeros, Theorem 8 represents the proof of GRH for a wide class of functions.

  \bigskip  

\textbf{Acknowledgement}: The author is very much thankful to Florin Alan Muscutar for his contribution with computer generated graphics.

\bigskip

\textbf{\textbf{References}}

\bigskip

[1] Cahen, E. 1894. Sur la fonction $\zeta (s)$ de Riemann et sur des
fonctions analogues, Ann. Ecole Norm.,$3^{e}$ serie, t. 11 

\bigskip

[2] Cahen, E. 1918. Sur les series de Dirichlet, Compte rendue, t. 166

 \bigskip
 
[3] Hadamard, J. 1896. Sur la distribution des zeros de la fonction $%
\zeta(s)$ de Riemann et ses consequences arithmetiques, Bull. Soc. math., t. 24

\bigskip

[4] Hadamard, J.1908. Sur les series de Dirichlet, Rendiconti di Palermo, t. 25 

 \bigskip

[5] Landau, E. 1907. Ueber die Multiplication Dirichletscher Reihen,
rendiconti di Palermo, t. 24 

 \bigskip
 
[6] Landau, E. 1909. Ueber das Konvergenzproblem der Dirichletschen Reihen, Rendiconti di Palermo, t. 28

 \bigskip
 
[7] Landau, E. 1917.  Neuer Beweis eines Haptsatzes aus der Theorie der Dirichletschen Reihen, Leipziger Ber.,t. 69

 \bigskip
 
[8] Landau, E. 1921. Ueber die gleichmassige Konvergenz Dirichletscher Reihen,
Math. Zeitschrift, t. 11 

\bigskip
   
[9] Bohr, H. 1913(a). Ueber die gleichmassige Konvergenz Dirichletscher Reihen ,
J. fuer Math, t. 143

\bigskip

[10] Bohr, H. 1913(b). Einige Bemerkungen ueber das konvergenzproblem
Dirichletscher Reihen, Rendicnti di Palermo, t. 37

\bigskip

[11] Bohr, H. 1913(c). Darstellung der gleichmassigen Konvergenzabszisse einer Dirichletschen Reihe als Funktionen der Koeffizienten der Reihe, Archiv. der Math. und Phys., t. 21

\bigskip

[12]  Bohr, H. 1913(d). Zur Theorie der allgemeinen Dirichletschen Reihen, Math. Annalen, t. 79

\bigskip

[13]  Hardi, G.-H. 1911. The multiplication of Dirichlet's series, London math.
Soc., t. 10

\bigskip

[14]  Hardy, G. H. and Riesz, M. 1915. The General Theory of Dirichlet's Series,
Cambridge University Press

\bigskip

[15]  Kojima, T. 1914. On the convergence abscissa of general Dirichlet's
series, Tohoku J., t. 6

\bigskip

[16]  Kojima, T. 1916.  Note on the convergence abscissa of Dirichlet's series, Tohoku J., t. 9 

\bigskip

[17]  Kuniyeda, M. 1916. Uniform convergence abscissa of Dirichlet's series, Quat.J.,t. 47

\bigskip

[18] Valiron, G. 1926. Theorie generale des series de Dirichlet, Memorial des
sciences mathematiques, fascicuule 17, 1-56 

\bigskip

[19] Jiarong, Y., 1963, Borel's line of entire functions represented by Laplace-Stieltjes transformations, Acta Math. Sinica, 13, 471-484

\bigskip

[20] Kong, Y. Y. and Daochun, S., 2008. On type-function and growth of Laplace-Stieltjes transformations convergent in the right half plane, 2008. Adv. math, 37, 197-205 

 \bigskip
 
[21] Luo, X. and Kong, Y. Y.,2012. On the order and type of Laplace-Stieltjes transforms of slow growth, Acta Math. Sinica, 32A, 601-607
 
\bigskip

[22] Kong, Y. Y, and Yang, Y., 2014. On the growth properties of the Laplace-Stieltjes transform,  Complex Var. Elliptic Equ. 59(4), 553-563

 \bigskip
 
[23] Singhal, C. and Srivastava, G. S., 2015.  On the approximation of an analytic function represented by Laplace-Stieltjes transformation, Anal. Theory Appl. 31, 407-420

 \bigskip
 
[24] Srivastava, G. S. and Singhal, C., 2015.  On the generalized order and the generalized type of Laplace-Stieltjes transformation convergent in the right half plane, Global journal of Pure and Applied Mathematics, 11, 469-477 

\bigskip

[25] Kamthan, P. K. and Shing Gautam, S. K. 1975. Bases in a certain space of functions analytic in the half plane, Indian J. Pure Appl. Math., 6, 1066-1075

\bigskip

[26] Khoi, L. H., 1998, Holomorphic Dirichlet series in the half plane

 \bigskip
 
[27] Defant, A., Garcia, D, Maestre, M., Perez-Garcia, D., 2008, Bohr's strip for vector valued Dirichlet series, Math. Ann. 342, no.3, 533-555

 \bigskip

[28] Bonet, J. 2015, Abscissas of weak convergence of vector valued Dirichlet series, arXiv: 1502.00418v1

\bigskip

[29] Srivastava, B. L. 1983. A Study of Spaces of Certain Classes of Vector Valued Dirichlet Series, Theis, I. I. T. Kanpur

\bigskip

[30] Srivastava, G. S. and Sharma, A. 2011, Bases in the space of vector valued analytic Dirichlet series, International J. Pure Appl. Math., 7, 993-1000

\bigskip

[31] Srivastava, G. S. and Sharma, A. 2012. Spaces of entire functions represented by vector valued Dirichlet series of slow growth, Acta Universitatis Sapientiae, Mathematica, 4, 2, 154-167 

\bigskip
 
[32] Ghisa, D. 2013. Fundamental Domains and the Riemann Hypothesis, Lambert
Academic Publishing

 \bigskip
 
[33] Ghisa, D. 2014. On the Generalized Riemann Hypothesis, Complex Analysis and
Potential Theory with Applications, Cambridge Scientific Publishers, 77-94

\bigskip

[34] Ghisa, D. 2016a. On the Generalized Riemann Hypothesis II,IJSIMR,4, 46-55

\bigskip

[35] Ghisa, D. 2016b. Fundamental Domains and Analytic Continuation of General
Dirichlet Series, BJMCS

\bigskip

[36] Andreian-Cazacu, C. and Ghisa, D. 2009. Global Mapping Properties of Analytic Functions,
Proceedings of 7-th ISAAC Congress, London, U. K., 3-12

\bigskip
 
[37] Andreian-Cazacu, C. and Ghisa, D. 2011. Fundamental Domains of Gamma and Zeta Functions, International Journal of Mathematics and Mathematical Sciences, http://dx.doi.org/10.1155/2011/985323, 21 pages
 \bigskip

[38] Speiser, A. 1934. Geometrisches zur Riemannschen Zetafunction, Math. Ann. 110, 514-521

\bigskip

[39] Arias-de-Reina, J. 2003. X-Ray of Riemann's Zeta-Function, 
ArXiv:math/0309433, 42 pages

\bigskip

[40] Wegert, E. 2012. Visual Complex Functions, Birkhauser

 \bigskip
 
[41] Wegert, E. 2016. Visual Exploration of Complex Functions, Springer Proceedings in Mathematics and Statistics, Volume 177

\bigskip

[42] Barza, I., Ghisa, D. and Muscutar, F. A. 2014.  On the Location of the Zeros of the Derivatives of Dirichlet L-functions, Annals of the University  of Bucharest (mathematics series) 5(LXIII), 21-31

\bigskip
 
[43] Cao-Huu, T., Ghisa, D. and Muscutar, F. A. 2016. Multiple Solutions of Riermann Type of Functional Equations, British Journal of Mathematics and Computer Science 17(3) 1-12

\bigskip

\end{document}